\documentclass[11pt,reqno]{amsart}

\usepackage{amsthm,amsmath,amssymb}
\usepackage{graphicx}

\def\C{{\mathcal C}}
\def\E{{\mathcal E}}
\def\eps{\varepsilon}
\def\R{\mathbb{R}}
\def\Ray{\mathcal{R}}
\def\T{{\mathbb T}}
\def\oeps{{\mathcal{O}_\eps}}
\def\bal{\begin{aligned}}
\def\eal{\end{aligned}}

\newtheorem{lemma}{Lemma}[section]

\newtheorem*{theorem}{Theorem}
\newtheorem*{basic}{Basic properties of $\E$}
\theoremstyle{definition}

\title[The boundary of the attainable set]{On the boundary of the attainable set\\ of the Dirichlet spectrum}

\author[Lorenzo Brasco]{Lorenzo Brasco}
\address{Laboratoire d'Analyse, Topologie, Probabilit\'es UMR6632, Universit\'e Aix--Marseille 1, CMI 39, Rue Fr\'ed\'eric Joliot Curie, 13453 Marseille Cedex 13, France}
\email{brasco@cmi.univ-mrs.fr}
\author{Carlo Nitsch}
\address{Dipartimento di Matematica e Applicazioni, Universit\`a di Napoli ``Federico II'', Complesso di Monte S. Angelo, Via Cintia, 80126 Napoli, Italy}
\email{c.nitsch@unina.it}
\author{Aldo Pratelli}
\address{Dipartimento di Matematica ``F. Casorati'', Universit\`a di Pavia, via Ferrata 1, 27100 Pavia, Italy}
\email{aldo.pratelli@unipv.it}
\numberwithin{equation}{section}

\begin{document}

\begin{abstract}
Denoting by $\E\subseteq \R^2$ the set of the pairs $\big(\lambda_1(\Omega),\,\lambda_2(\Omega)\big)$ for all the open sets $\Omega\subseteq\R^N$ with unit measure, and by $\Theta\subseteq\R^N$ the union of two disjoint balls of half measure, we give an elementary proof of the fact that $\partial\E$ has horizontal tangent at its lowest point $\big(\lambda_1(\Theta),\,\lambda_2(\Theta)\big)$.
\end{abstract}

\maketitle

\section{Introduction}

Given an open set $\Omega\subseteq \R^N$ with finite measure, its Dirichlet-Laplacian spectrum is given by the numbers $\lambda>0$ such that the boundary value problem
\[
-\Delta u=\lambda\, u\ \mbox{ in }\Omega,\qquad u=0\ \mbox{ on } \partial\Omega,
\]
has non trivial solutions. Such numbers $\lambda$ are called {\it eigenvalues of the Dirichlet-Laplacian in $\Omega$}, and form a discrete increasing sequence $0<\lambda_1(\Omega)\le\lambda_2(\Omega)\le\lambda_3(\Omega)\dots$, diverging to $+\infty$ (see~\cite{He}, for example). In this paper, we will work with the first two eigenvalues $\lambda_1$ and $\lambda_2$, for which we briefly recall the variational characterization: introducing the {\it Rayleigh quotient} as
\[
\Ray_\Omega(u)=\frac{\|\nabla u\|^2_{L^2(\Omega)}}{\|u\|^2_{L^2(\Omega)}},\qquad u\in H^1(\Omega)\,,
\]
the first two eigenvalues of the Dirichlet-Laplacian satisfy
\begin{gather*}
\lambda_1(\Omega)=\min\Big\{\,\Ray_\Omega(u) :\, u\in H^1_0(\Omega)\setminus\{0\}\Big\}\,,\\
\lambda_2(\Omega)=\min\left\{\Ray_\Omega(u)\, :\, u\in H^1_0(\Omega)\setminus\{0\}, \, \int_\Omega u(x)\, u_1(x)\, dx=0\right\}\,,
\end{gather*}
where $u_1$ is a first eigenfunction.\par
We are concerned about the \emph{attainable set} of the first two eigenvalues $\lambda_1$ and $\lambda_2$, that is,
\[
\E:=\Big\{\big(\lambda_1(\Omega),\lambda_2(\Omega)\big)\in\R^2\, :\, \big|\Omega\big|=\omega_N\Big\}\,,
\]
where $\omega_N$ is the volume of the ball of unit radius in $\R^N$. Of course, the set $\E$ depends on the dimension $N$ of the ambient space. The set $\E$ has been deeply studied (see for instance~\cite{AH,BBF,KW}); an approximate plot is shown in Figure~\ref{figure1}. Let us recall now some of the most important known facts. In what follows, we will always denote by $B$ a ball of unit radius (then, of volume $\omega_N$), and by $\Theta$ a disjoint union of two balls of volume $\omega_N/2$.
\begin{basic}
The attainable set $\E$ has the following properties: 
\begin{enumerate}
\item[(i)] for every $(\lambda_1,\lambda_2)\in\E$ and every $t\geq 1$, one has $(t\, \lambda_1,t\, \lambda_2)\in\E$;
\item[(ii)]
\[
\E\subseteq \bigg\{ x\geq \lambda_1(B),\, y \geq \lambda_2(\Theta),\,  1\leq \frac y x\leq \frac{\lambda_2(B)}{\lambda_1(B)}\, \bigg\}\,\hbox{;}
\]
\item[(iii)] $\E$ is \emph{horizontally} and \emph{vertically convex}, i.e., for every $0\leq t \leq 1$
\begin{align*}
(x_0,y),(x_1,y)\in\E &\Longrightarrow \big((1-t)x_0+tx_1,y\big)\in\E\,, \\
(x,y_0),(x,y_1)\in\E &\Longrightarrow \big(x,(1-t)y_0+ty_1\big)\in\E\,.
\end{align*}
\end{enumerate}
\end{basic}
The first property is a simple consequence of the scaling property $\lambda_i(t\,\Omega) = t^{-2} \lambda_i (\Omega)$, valid for any open set $\Omega\subseteq\R^N$ and any $t>0$. The second property is true because, for every open set $\Omega$ of unit measure, the Faber--Krahn inequality ensures $\lambda_1(\Omega)\geq \lambda_1(B)$, the Krahn--Szego inequality (see~\cite{Hong,Kr,Po}) ensures $\lambda_2(\Omega)\geq \lambda_2(\Theta)=\lambda_1(\Theta)$, and a celebrated result by Ashbaugh and Benguria (see~\cite{AshBen}) ensures
\[
1\leq \frac{\lambda_2(\Omega)}{\lambda_1(\Omega)}\le \frac{\lambda_2(B)}{\lambda_1(B)}.
\]
Finally, the third property is proven in~\cite{BBF}. It has been conjectured also that the set $\E$ is convex, as it seems reasonable by a numerical plot, but a proof for this fact is still not known.\par
\begin{figure}[tb]
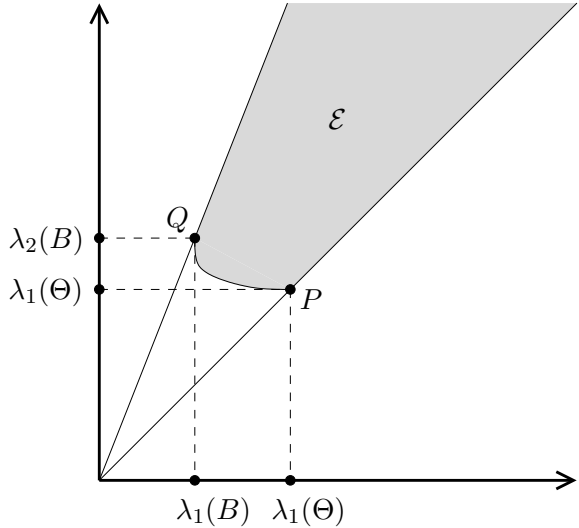
\caption{The attainable set $\E$}\label{figure1}
\end{figure}
Thanks to the above listed properties, the set $\E$ is completely known once one knows its ``lower boundary''
\[
\C := \Big\{ \big(\lambda_1,\,\lambda_2\big)\in \overline{\E} :\, \forall\, t<1,\,\big(t \lambda_1,\,t \lambda_2\big)\notin \E\Big\}\,,
\]
therefore studying $\E$ is equivalent to study $\C$. Notice in particular that $\partial\E$ consists of the union of $\C$ with the two half-lines 
\[
\big\{(t,t):\, t\geq \lambda_1(\Theta)\big\}\qquad\mbox{ and }\qquad\left\{\left(t,\frac{\lambda_2(B)}{\lambda_1(B)}\,t\right):\, t\geq \lambda_1(B)\right\}.
\]
Let us call for brevity $P$ and $Q$ the endpoints of $\C$, that is, $P\equiv \big(\lambda_1(\Theta),\lambda_2(\Theta)\big)$ and $Q\equiv \big(\lambda_1(B),\lambda_2(B)\big)$.\par

The plot of the set $\E$ seems to suggest that the curve $\C$ reaches the point $Q$ with vertical tangent, and the point $P$ with horizontal tangent. In fact, Wolf and Keller in~\cite[Section 5]{KW} proved the first fact, and they also suggested that the second fact should be true, providing a numerical evidence. The aim of the present paper is to give a short proof of this fact.
\begin{theorem}
For every dimension $N\geq 2$, the curve $\C$ reaches the point $P$ with horizontal tangent.
\end{theorem}
The rest of the paper is devoted to prove this result: the proof will be achieved by exhibiting a suitable family $\{\widetilde\Omega_\eps\}_{\eps>0}$ of deformations of $\Theta$ having measure $\omega_N$ and such that
\begin{equation}\label{gothere}
\lim_{\eps\to 0} \frac{\lambda_2(\widetilde\Omega_\eps)-\lambda_2(\Theta)}{\lambda_1(\Theta)-\lambda_1(\widetilde\Omega_\eps)} =0\,.
\end{equation}

\section{Proof of the Theorem}

Throughout this section, for any given $x=(x_1,...,x_N)\in \R^N$, we will write $x=(x_1,x')$ where $x_1\in \R$  and $x'\in\R^{N-1}$.\par
We will make use of the sets $\{\Omega_\eps\}\subseteq\R^N$, shown in Figure~\ref{fig:omeps}, defined by
\[\begin{split}
\Omega_\eps :=& \Big\{ (x_1,x')\in \R^+\times\R^{N-1}:\, (x_1-1+\eps)^2+ |x'|^2 <1 \Big\} \\
&\cup \Big\{ (x_1,x')\in \R^-\times\R^{N-1}:\, (x_1+1-\eps)^2+ |x'|^2 <1 \Big\}\\
=:&\Omega_\eps^+\cup \Omega_\eps^-\,.
\end{split}\]
for every $\eps>0$ sufficiently small. The sets $\widetilde\Omega_\eps$ for which we will eventually prove~(\ref{gothere}) will be rescaled copies of $\Omega_\eps$, in order to have measure $\omega_N$.\par
To get our thesis, we need to provide an upper bound to $\lambda_1(\Omega_\eps)$ and an upper bound to $\lambda_2(\Omega_\eps)$; this will be the content of Lemmas~\ref{upper1} and~\ref{upper2} respectively.
\begin{figure}[htbp]
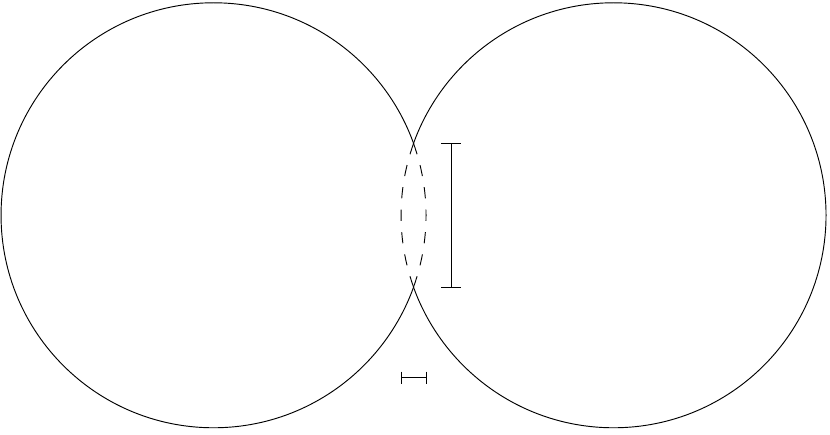
\caption{The sets $\Omega_\eps=\Omega^+_\eps \cup \Omega_\eps^-$}\label{fig:omeps}
\end{figure}
\begin{lemma}\label{upper1}
There exists a constant $\gamma_1>0$ such that for every $\eps\ll 1$ it is
\begin{equation}
\label{seconda}
\lambda_1(\Omega_\eps)\leq \lambda_1(B)-\gamma_1\,\eps^{N/2}.
\end{equation}
\end{lemma}
\begin{proof}
Let $B_\eps$ be the ball of unit radius centered at $(1-\eps,0)$, so that $B_\eps\subseteq \Omega_\eps$ and in particular $\Omega_\eps^+=B_\eps \cap \{x_1>0\}$. Let also $u$ be a first Dirichlet eigenfunction of $B_\eps$ with unit $L^2$ norm, and denote by $\T$ the region (shaded in Figure~\ref{lastest}) bounded by the right circular conical surface $\{\sqrt{2\eps-\eps^2} -x_1- |x'|=0 \}$ and by the plane $\{x_1=0\}$. 

\begin{figure}[htbp]
\input{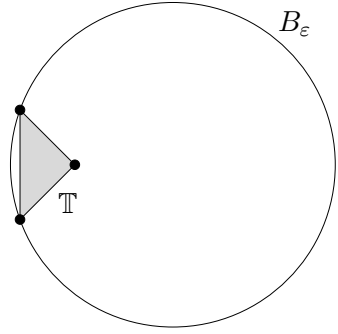}
\caption{The ball $B_\eps$ and the cone $\T$ (shaded) in the proof of Lemma~\ref{upper1}}\label{lastest}
\end{figure}

Since the normal derivative of $u$ is constantly $\kappa$ on $\partial B_\eps^+$, we know that
\begin{equation}\label{sim1}
Du (x_1,x') = Du ( 0,x') + O\big(\sqrt\eps\big) = \big(\kappa, 0\big) + O\big(\sqrt\eps\big) \quad \hbox{on $\T$}\,.
\end{equation}
Let us now define the function $\tilde u:\Omega_\eps^+\to \R$ as
\[
\tilde u(x_1,x')  := \left\{\begin{array}{ll}
u(x_1,x') & \hbox{if $(x_1,x')\notin\T$}\,, \\
u(x_1,x') + \bal\frac\kappa 2\eal \, \Big(\sqrt{2\eps-\eps^2}-x_1 - |x'|  \Big) & \hbox{if $(x_1,x')\in\T$}\,.
\end{array}\right.
\]
It is immediate to observe that $\tilde u=u$ on the surface $\left\{\sqrt{2\eps-\eps^2} -x_1- |x'|=0\right\}\cap\{ x_1 >0 \}$, so that $\tilde u \in H^1(\Omega_\eps^+)$. Notice that $\tilde u\notin H^{1}_0(\Omega_\eps^+)$ since $\tilde u$ does not vanish on $\{x_1=0\}\cap \partial \Omega^+_\eps$. By construction and recalling~(\ref{sim1}),
\begin{equation}\label{sim2}
D\tilde u (x_1,x') = D u (x_1,x') + \Big( -\frac \kappa 2 , -\,\frac\kappa 2 \, \frac {x'} {|x'|}\Big)
= \Big(\frac \kappa 2, - \,\frac \kappa 2 \, \frac {x'} {|x'|}\Big) + O\big(\sqrt\eps\big) \quad \hbox{on $\T$}\,.
\end{equation}
Since $\tilde u \geq u$ on $\Omega_\eps^+$, and recalling that $u\in H^1_0(B_\eps^+)$, one clearly has
\begin{equation}\label{squa1}
\int_{\Omega_\eps^+} \tilde u^2 dx \geq \int_{\Omega_\eps^+} u^2 dx= \int_{B_\eps^+} u^2 dx + O(\eps^{(N+5)/2})
= 1  + O(\eps^{(N+5)/2})\,,
\end{equation}
since the small region $B_\eps\setminus \Omega_\eps^+$ has volume $O(\eps^{(N+1)/2})$, and on this region $u=O(\eps)$.\par
On the other hand, comparing~(\ref{sim1}) and~(\ref{sim2}), one has
\[
\big|D \tilde u\big|^2\, =| D u \big|^2 - \frac {\kappa^2} 2 + O\big(\sqrt\eps\big)\quad \hbox{on $\T$}\,,
\]
and since the volume of $\T$ is $\frac{\omega_{N-1}}N\big(2\eps-\eps^2\big)^{N/2}$
we deduce
\begin{equation}\label{squa2}\begin{split}
\int_{\Omega_\eps^+}\big| D \tilde u\big|^2\, dx &=
\int_{\Omega_\eps^+}\big| D u \big|^2\, dx - \frac{\omega_{N-1}}N\big(2\eps-\eps^2\big)^{N/2}\,\bigg(\frac{\kappa^2}{2} + O(\sqrt \eps)\bigg)\\
&=\int_{\Omega_\eps^+}\big| D u \big|^2\, dx - \frac {\omega_{N-1}}N\kappa^2 2^{(N/2-1)}\eps^{N/2} + O(\eps^{(N+1)/2})\\
&=\int_{B_\eps^+} \big| D u \big|^2\, dx - C_N\kappa^2\eps^{N/2} + O(\eps^{(N+1)/2})\,,
\end{split}\end{equation}
where $C_N=\frac {\omega_{N-1}}N 2^{(N/2-1)}$.\par
Therefore, by~(\ref{squa1}) and~(\ref{squa2}) we obtain
\[
\begin{split}
\Ray_{\Omega_\eps^+} (\tilde u) = \frac{\bal\int_{\Omega_\eps^+} \big| D\tilde u\big|^2\, dx\eal}{\bal\int_{\Omega_\eps^+}\tilde u^2\, dx\eal} &\leq \Ray_{B_\eps^+} (u) - C_N\kappa^2 \eps^{N/2} + O(\eps^{(N+1)/2})\\
&= \lambda_1(B) - C_N\kappa^2 \eps^{N/2} + O(\eps^{(N+1)/2})\,.
\end{split}
\]
We can finally extend $\tilde u$ to the whole $\Omega_\eps$, simply defining $\tilde u(x_1,x')= \tilde u(|x_1|,x')$ on $\Omega_\eps^-$. By construction, $\tilde u\in H^1_0(\Omega_\eps)$, and
\[
\lambda_1(\Omega_\eps) \leq \Ray_{\Omega_\eps} (\tilde u) = \Ray_{\Omega_\eps^+} (\tilde u)
\leq \lambda_1(B) - C_N\kappa^2 \eps^{N/2} + O(\eps^{(N+1)/2})\,,
\]
so that~(\ref{seconda}) follows and the proof is concluded.
\end{proof}

\begin{lemma}\label{upper2}
There exists a constant $\gamma_2>0$ such that for every $\eps\ll 1$, it is
\begin{equation}
\label{prima}
\lambda_2(\Omega_\eps)\leq \lambda_1(B)+\gamma_2\,\eps^{(N+1)/2}.
\end{equation}
\end{lemma}
\begin{proof}
First of all, we start underlining that
\begin{equation}\label{estimateleq}
\lambda_2(\Omega_\eps)\leq \lambda_1(\Omega_\eps^+)\,;
\end{equation}
in fact if we define
\[
\tilde u(x_1,x'):=\left\{\begin{array}{ll}
u_\eps(x_1,x')\,,& \mbox{ if } x_1\in \Omega^+_\eps\,,\\
-u_\eps(-x_1,x')\,,& \mbox{ if } x_1\in \Omega^-_\eps\,,
\end{array}\right.
\]
then by construction it readily follows that $-\Delta \tilde u = \lambda_1(\Omega_\eps^+) \tilde u$. As a consequance $\lambda_1(\Omega_\eps^+)$ is an eigenvalue of $\Omega_\eps$, say $\lambda_1(\Omega_\eps^+)=\lambda_\ell(\Omega_\eps)$. Since $\Omega_\eps$ is connected and $\tilde u$ changes sign, it is not possible $\ell=1$, 
hence
$$\lambda_2(\Omega_\eps)\leq \lambda_\ell(\Omega_\eps)=\lambda_1(\Omega_\eps^+)\,.$$

It is then enough for us to estimate $\lambda_1(\Omega_\eps^+)$. To this aim, define the set
\[
\oeps :=\big\{(x_1,x')\in\Omega^+_\eps\, :\, x_1\ge \eps\big\},
\]
and take a Lipschitz cut-off function $\xi_\eps \in W^{1,\infty}(\Omega^+_\eps)$ such that
\begin{align*}
0\leq \xi_\eps\leq 1\ \hbox{on}\ \Omega^+_\eps\,, && \xi_\eps\equiv 1\  \hbox{on}\ \oeps\,, && 
\xi_\eps \equiv 0\ \hbox{on}\ \partial\Omega^+_\eps\cap \{x_1=0\}\,,
&& \|\nabla \xi_\eps\|_\infty\le L\, \eps^{-1}\,.
\end{align*}
As in Lemma~\ref{upper1}, let again $u$ be a first eigenfunction of the ball $B_\eps$ of radius $1$ centered at $(1-\eps,0)$ having unit $L^2$ norm, and define on $\Omega_\eps$ the function $\varphi=u\, \xi_\eps$. Since by construction $\varphi$ belongs to $H^1_0(\Omega_\eps)$, we obtain
\begin{equation}
\label{puthere}
\lambda_1(\Omega^+_\eps)\leq \Ray(\varphi, \Omega^+_\eps)
=\frac{\bal\int_{\Omega^+_\eps} \Big[|\nabla u|^2\, \xi_\eps^2+  |\nabla \xi_\eps|^2\, u^2+ 2\, u\, \xi_\eps\,\langle\nabla u,\nabla\xi_\eps \rangle\Big]\, dx\eal}{\bal\int_{\Omega^+_\eps} u^2\, \xi_\eps^2\, dx\eal}\,.
\end{equation}
We can start estimating the denominator very similarly to what already done in~(\ref{squa1}). Indeed, recalling that $\big|\Omega^+_\eps\setminus \oeps\big|= O(\eps^{(N+1)/2})$ and that in that small region $u=O(\eps)$, we have
\[
\int_{\Omega^+_\eps} u^2\, \xi_\eps^2\, dx
=\int_{B_\eps} u^2\, dx - \int_{B_\eps\setminus \Omega^+_\eps} u^2\, dx -\int_{\Omega_\eps^+\setminus \oeps} u^2 (1-\xi_\eps^2)\, dx = 1 +O(\eps^{(N+5)/2})\,.
\]
Let us pass to study the numerator: first of all, being $0\leq \xi_\eps\leq 1$ we have
\[
\int_{\Omega^+_\eps} |\nabla u|^2\, \xi_\eps^2\, dx\leq \int_{B_\eps} |\nabla u|^2 \, dx= \lambda_1(B)\,.
\]
Moreover,
\[
\int_{\Omega^+_\eps} |\nabla \xi_\eps|^2\, u^2\, dx=\int_{\Omega^+_\eps\setminus\mathcal{O}_\eps} |\nabla \xi_\eps|^2\, u^2\, dx\le \frac{L^2}{\eps^2}\, |\Omega^+_\eps\setminus\mathcal{O}_\eps|\, \|u\|^2_{L^\infty(\Omega_\eps^+\setminus \mathcal{O}_\eps)} = O(\eps^{(N+1)/2})\,,
\]
and in the same way
\[
\int_{\Omega^+_\eps} u\, \xi_\eps\,\langle \nabla u, \nabla\xi_\eps\rangle\, dx 
\leq \int_{\Omega^+_\eps\setminus\mathcal{O}_\eps} |u| \, |\nabla u|\,|\nabla \xi_\eps|\, dx 
= O(\eps^{(N+1)/2})\,.
\]
Summarizing, by~(\ref{puthere}) we deduce
\[
\lambda_1(\Omega^+_\eps)\leq \lambda_1(B) + O(\eps^{(N+1)/2})\,,
\]
thus by~(\ref{estimateleq}) we get the thesis.
\end{proof}


We are now ready to conclude the paper by giving the proof of the Theorem.

\begin{proof}[Proof of the Theorem.]
For any small $\eps>0$, we define $\widetilde \Omega_\eps = t_\eps \, \Omega_\eps$, where $t_\eps = \sqrt[N]{\omega_N/|\Omega_\eps|}$ so that $|\widetilde\Omega_\eps|=\omega_N$. Notice that
\[
|\Omega_\eps| = 2 \omega_N + O(\eps^{(N+1)/2})\,,
\]
thus $t_\eps = 2^{-1/N} + O(\eps^{(N+1)/2})$. Recalling the trivial rescaling formula $\lambda_i(t\Omega)=t^{-2}\lambda_i(\Omega)$, valid for any natural $i$, any positive $t$ and any open set $\Omega$, we can then estimate by Lemma~\ref{upper1} and Lemma~\ref{upper2}
\begin{gather*}
\lambda_1(\widetilde\Omega_\eps) = \left(\frac{|\Omega_\eps|}{\omega_N}\right)^{2/N} \lambda_1(\Omega_\eps)
\leq 2^{2/N}\lambda_1(B) -2^{2/N}\gamma_1\eps^{N/2} + O (\eps^{(N+1)/2})\,, \\
\lambda_2(\widetilde\Omega_\eps) = \left(\frac{|\Omega_\eps|}{\omega_N}\right)^{2/N} \lambda_2(\Omega_\eps)
\leq 2^{2/N}\lambda_1(B)  + O (\eps^{(N+1)/2})\,.
\end{gather*}
Since $\lambda_1(\Theta)=\lambda_2(\Theta) = 2^{2/N}\lambda_1(B)$, the two above estimates give
\[
\lim_{\eps\to 0} \frac{\lambda_2(\widetilde\Omega_\eps)-\lambda_2(\Theta)}{\lambda_1(\Theta)-\lambda_1(\widetilde\Omega_\eps)} =0\,,
\]
which as already noticed in~(\ref{gothere}) implies the thesis.
\end{proof}

\subsection*{Acknowledgements}
The three authors have been supported by the ERC Starting Grant n. 258685; L. B. and A. P. have been supported also by the ERC Advanced Grant n. 226234.

\end{document}